\documentclass[12pt]{amsart}

\usepackage[dvipsnames]{xcolor}
\usepackage{graphicx}
\usepackage{hyperref}
\usepackage{comment}
\usepackage{bbm}
\usepackage{dsfont}
\usepackage{amsmath}
\usepackage{color}
\usepackage{caption}
\usepackage{subcaption}
\usepackage[T1]{fontenc}
\usepackage{mathtools}
\usepackage{amsfonts,amssymb,amscd}

\calclayout

\newtheorem{thm}{Theorem}[section]
\newtheorem{prop}[thm]{Proposition}

\theoremstyle{definition}
\newtheorem{definition}[thm]{Definition}

\newtheorem{notation}[thm]{Notation}

\theoremstyle{remark}

\newtheorem{rem}[thm]{Remark}

\theoremstyle{question}

\newtheorem{thmx}{Theorem}

\numberwithin{equation}{section}

\begin{document}


\title{Interpolation of Power Mappings}


\author{Jack Burkart}
\address{Department of Mathematics, University of Wisconsin-Madison}
\email{burkart2@wisc.edu}


\author{Kirill Lazebnik}
\address{Department of Mathematics, University of North Texas}
\email{Kirill.Lazebnik@unt.edu}
\urladdr{https://math.unt.edu/~lazebnik/}




\begin{abstract} 




Let $(M_j)_{j=1}^\infty\in\mathbb{N}$ and $(r_j)_{j=1}^\infty\in\mathbb{R}^+$ be increasing sequences satisfying some mild rate of growth conditions. We prove that there is an entire function $f: \mathbb{C} \rightarrow\mathbb{C}$ whose behavior in the large annuli $\{ z\in\mathbb{C} : r_{j}\cdot\exp(\pi/M_{j})\leq|z|\leq r_{j+1}\}$ is given by a perturbed rescaling of $z\mapsto z^{M_j}$, such that the only singular values of $f$ are rescalings of $\pm r_j^{M_j}$. We describe several applications to the dynamics of entire functions.


\end{abstract}

\maketitle


\tableofcontents

\section{Introduction}
\label{introduction}

The problem of constructing analytic functions with prescribed geometry arises in areas across function theory. One general approach consists of defining a convergent infinite product, which has the advantage, for instance, of yielding an explicit expression for the zeros of the function. We mention several dynamical applications of this approach: the first construction of an entire function with a wandering domain \cite{Bak76}, constructions of entire functions of slow growth with Julia set = $\mathbb{C}$ \cite{MR1804945}, in the study of escaping sets of entire functions \cite{MR3095155}, the construction of transcendental entire functions with Julia sets of Hausdorff dimension $1$ \cite{Bis18}, and in the dynamics of sine maps \cite{MR4181721}.



A general principle used in the above works is that the behavior of an infinite product is dominated within certain regions by certain factors in the product. Our main result similarly proves the existence of an entire function whose behavior is dominated within certain annuli by certain prescribable monomials. One advantage of our approach, however, is a precise description of the singular values of the entire function. This information is in general difficult to glean from an infinite product, and is often crucial in applications. Before describing further motivation and applications, we state our main result. First we will need the following definition, which will serve as an assumption on the growth rate of the degree of the aforementioned monomials. 

\begin{definition}\label{permissible} Let $(M_j)_{j=1}^\infty \in \mathbb{N}$ be increasing, and $(r_j)_{j=1}^\infty \in \mathbb{R}^+$. We say that $(M_j)_{j=1}^\infty$, $(r_j)_{j=1}^\infty$ are \emph{permissible} if \begin{gather} r_{j+1} \geq  \exp\left(\pi\big/M_j\right) \cdot r_j \textrm{ for all } j\in\mathbb{N}\textrm{, } r_j\xrightarrow{j\rightarrow\infty}\infty\textrm{, and } \sup_j\frac{M_{j+1}}{M_j}<\infty.  \end{gather} 
\end{definition}

\noindent With Definition \ref{permissible}, we can state our main result:


\begin{thmx}\label{mainthm} Let  $(M_j)_{j=1}^\infty$, $(r_j)_{j=1}^\infty$ be permissible, $r_0:=0$ and $c\in\mathbb{C}^\star:=\mathbb{C}\setminus\{0\}$.  Set
\begin{equation}  c_1:=c\textrm{, and } c_j:=c_{j-1}\cdot r_{j-1}^{M_{j-1}-M_{j}} \textrm{ for } j\geq2.   \end{equation} 
Then there exists an entire function $f: \mathbb{C} \rightarrow\mathbb{C}$ and a homeomorphism $\phi: \mathbb{C} \rightarrow \mathbb{C}$ such that \begin{gather} f\circ\phi(z)=c_jz^{M_j} \textrm{ for } r_{j-1}\cdot\exp(\pi/M_{j-1}) \leq |z| \leq r_j\textrm{, } j\in\mathbb{N}.  \end{gather} Moreover, if $\sum_{j=1}^\infty M_j^{-1}<\infty$, then $|\phi(z)/z - 1|\rightarrow 0$ as $z\rightarrow\infty$. The only singular values of $f$ are the critical values $(\pm c_jr_j^{M_j})_{j=1}^\infty$.
\end{thmx}

\begin{rem}\label{Theorem_remark} The homeomorphism $\phi$ is a \emph{$K$-quasiconformal homeomorphism} (see Definition \ref{qc_mapping}), where $K$ depends only on $\sup_j(M_{j+1}/M_j)$. The conclusion $|\phi(z)/z-1|\rightarrow0$ is deduced from the Teichm\"uller-Wittich-Belinskii Theorem (see Theorem \ref{TWB}), and in many applications, we can even deduce uniform estimates $||\phi(z)/z-1||_{L^\infty(\mathbb{C})}<\varepsilon$ (see Section \ref{dynamical_application_section}). We further remark that a precise description of the critical points and zeros of $f$ are also given (see Proposition \ref{critical_point_listing} and Remark \ref{zero_listing}), up to the perturbation $\phi$. The ``scaling'' constants $(c_j)_{j=1}^\infty$ ensure that $z\mapsto c_jz^{M_j}$ and $z\mapsto c_{j+1}z^{M_{j+1}}$ both map $|z|=r_j$ to the same scale. Lastly, we comment that if we replace the condition $r_j\rightarrow\infty$ of Definition \ref{permissible} with the condition $r_j\rightarrow r_\infty<\infty$, we obtain a result similar to Theorem \ref{mainthm}, but with the domain of $f$ equal to $r_\infty\cdot\mathbb{D}$ rather than $\mathbb{C}$ (see Theorem \ref{mainthm2} in Section \ref{Entire_function_section}).


\end{rem}

As indicated in Remark \ref{Theorem_remark}, our methods rely on \emph{quasiconformal surgery}, a collection of techniques to which we refer  to \cite{FH09} for a survey. Among these techniques, there are at least two distinct approaches both termed \emph{quasiconformal surgery}. The first consists in  the construction of a quasiregular function $g$ possessing a  $g$-invariant Beltrami coefficient $\mu$. The integrating map for $\mu$ then conjugates $g$ to a holomorphic function. This approach appears first in \cite{MR819553} (to the best of the authors' knowledge), and is the most common use of the Measurable Riemann Mapping Theorem in complex dynamics, as it is inherently dynamical.

A different approach also termed \emph{quasiconformal surgery} consists of constructing a quasiregular function $g$ which in turn yields a holomorphic function $g\circ\phi^{-1}$, where $\phi$ is the integrating map for the Beltrami coefficient $g_{\overline{z}}/g_z$. This is the approach used in the present work, and has long found fundamental application in the type problem and value distribution theory (see, for instance, \cite{MR869798}, \cite{MR853888}), and, surprisingly, has found recent application in complex dynamics despite the lack of a conjugacy between $g$ and $g\circ\phi^{-1}$. Indeed, this approach was used in \cite{MR3316755} in settling a long-standing question about the existence of wandering domains for functions with bounded singular set (see also \cite{MR3339086}, \cite{MR4041106}, \cite{2019arXiv190410086L}, \cite{MR4008367}).

A general difficulty in this approach lies in proving that $\phi$ is only a small perturbation of the identity, which may be deduced, for instance, from showing that $g$ is holomorphic (and hence $g_{\overline{z}}/g_z=0$) except on a very small set. To this end, the work \cite{MR3316755} introduced a technique termed \emph{quasiconformal folding}, which has found many recent applications in complex dynamics and in function theory more broadly (see, for instance, \cite{MR3232011}, \cite{MR3420484}, \cite{2016arXiv161006278R}, \cite{MR3525384}, \cite{MR3701653}, \cite{MR3579902}, \cite{MR4023391}, \cite{MR4174037}). Our main methods of proof for Theorem \ref{mainthm} are influenced by this technique, but we emphasize that the present work does not rely directly on the results of \cite{MR3316755}, and indeed may be read and understood independently of the aforementioned works.


To conclude the Introduction, we remark on several applications of Theorem \ref{mainthm} and briefly describe the proof. In the present manuscript, we briefly present in Section \ref{dynamical_application_section} how Theorem \ref{mainthm} yields a robust approach to the  construction of entire functions with multiply connected wandering domains (first constructed in \cite{Bak76} - see also \cite{MR796748}, \cite{MR1278933}, \cite{MR2458806}, \cite{MR2900609}, \cite{MR2900166}, \cite{MR3149847}). In a companion manuscript \cite{BL_inprep}, we show how Theorem \ref{mainthm} gives a different approach to the result of \cite{Bis18} on existence of transcendental entire functions with Julia sets of Hausdorff dimension $1$ (see also \cite{MR4340830}). We also show in \cite{BL_inprep} how Theorem \ref{mainthm} yields an answer to Question 9.5 of \cite{MR3988603} on whether there can exist multiply connected wandering domains with infinite inner connectivity and uncountably many singleton complementary components.


The main step in the proof of Theorem \ref{mainthm} is in finding an efficient interpolation between the mapping $z\mapsto c_jz^{M_j}$ on $|z|=r_j$ and the mapping $z\mapsto c_{j+1}z^{M_{j+1}}$ on $|z| = r_j\cdot\exp(\pi/M_{j})$. Indeed, once this interpolation is found, one may define a quasiregular function $g: \mathbb{C}\rightarrow\mathbb{C}$ by $g(z):=c_jz^{M_j}$ in the large annuli $r_{j-1}\cdot\exp(\pi/M_{j-1}) \leq |z| \leq r_j$ and the above interpolation in the ``leftover'' thin annuli $r_{j}\leq |z| \leq r_j\cdot\exp(\pi/M_{j})$. The Measurable Riemann Mapping Theorem is then applied to the Beltrami coefficient $g_{\overline{z}}/g_z$ to produce a quasiconformal mapping $\phi$ such that $f:=g\circ\phi^{-1}$ is the entire function of Theorem \ref{mainthm}. We describe in detail the aforementioned interpolation in Section \ref{interpolation_section}.

We now briefly outline the paper. After collecting a couple preliminary results we will need in Section \ref{preliminaries}, we detail the specifics of the main interpolation in Section \ref{interpolation_section} where the primary contributions of the present work are contained. Section \ref{Entire_function_section} applies the results of Section \ref{interpolation_section} to build the entire function of Theorem \ref{mainthm}. In Section \ref{dynamical_application_section}, we consider dynamical applications of Theorem \ref{mainthm}.

{\subsection*{Acknowledgements}
\addtocontents{toc}{\protect\setcounter{tocdepth}{1}}

The authors would like to thank both anonymous referees for their suggestions which led to an improved version of the paper. }

\section{Preliminaries}
\label{preliminaries}

\begin{definition}\label{qc_mapping} An orientation-preserving homeomorphism $\phi: \mathbb{C}\rightarrow\mathbb{C}$ is said to be a \emph{quasiconformal mapping} if $\phi \in W^{1,2}_{\textrm{loc}}(\mathbb{C})$ and $||\phi_{\overline{z}}/\phi_z||_{L^\infty(\mathbb{C})} \leq k$ for some $0\leq k<1$.
\end{definition}

\begin{rem} We refer the reader to \cite{MR0344463} for a detailed study of quasiconformal mappings. We will assume a familiarity with the basic theory in what follows. We remark that a \emph{quasiregular} mapping is one which may be represented by a composition $f\circ\phi$, where $f$ is holomorphic, and $\phi$ is quasiconformal (see \cite{MR0344463}, Chapter VI). 
\end{rem}

\begin{notation} We abbreviate piecewise linear by PWL. Given a quasiregular map $f$, we will denote the dilatation constant of $f$ by $K(f)$, where $1\leq K(f)<\infty$. We denote $k(f):=(K(f)-1)/(K(f)+1)$, and occasionally we will use the notation $A(r_1, r_2):=\{ z \in \mathbb{C} : r_1 < |z| < r_2 \}$.
\end{notation}

We record a Theorem due to Teichm\"uller, Wittich, and Belinskii (Theorem \ref{TWB} below). As already mentioned in the Introduction, we will use this result to deduce the conclusion $|\phi(z)/z-1|\rightarrow0$ of Theorem \ref{mainthm}. The statement of the result is taken from Theorem 6.1 of \cite{MR0344463}, to which we refer for the relevant bibliography. We note that in Theorem \ref{TWB}, $\textrm{d}A(z)$ refers to area measure. Before stating the result, we recall two definitions which will appear in the Theorem.

\begin{definition}\label{dilatation_definition}  Let $\psi: \mathbb{C}\rightarrow\mathbb{C}$ be a quasiconformal mapping. The \emph{dilatation quotient} of $\psi$ is defined by \begin{equation} D(z):=\frac{|\psi_z(z)|+|\psi_{\overline{z}}(z)|}{|\psi_z(z)|-|\psi_{\overline{z}}(z)|}. \end{equation} The quantity $D(z)$ is defined for a.e. $z$ for a quasiconformal mapping $\psi$, and satisfies the relation $K(f)=\sup_{z\in\mathbb{C}} D(z)$ (see \cite{MR0344463}, Section IV.1.5).   
\end{definition}

\begin{definition} Let $\psi$ be a quasiconformal mapping defined in a neighborhood of a point $z_0$. The map $\psi$ is said to be \emph{conformal} at $z_0$ if the limit \begin{align} \lim_{z\rightarrow z_0}\frac{\psi(z)-\psi(z_0)}{z-z_0} \end{align} exists, in which case we denote the limit by $\psi_z(z_0)$. 
\end{definition}


\begin{thm}\label{TWB} Let $\psi$ be a $K$-quasiconformal mapping of the finite plane onto itself with $\psi(0)=0$ and \begin{equation} I(r):=\frac{1}{2\pi}\int_{|z|<r} \frac{D(z)-1}{|z|^2}\emph{d}A(z) < \infty \textrm{ for some } r < \infty. \end{equation} Then $\psi$ is conformal at $z=0$ and \begin{equation}\label{conclusion_of_TWB} \left| \frac{\psi(z)}{z} - \psi_z(0) \right| < |\psi_z(0)|\varepsilon(|z|),\hspace{3mm}   \varepsilon(|z|)\xrightarrow{|z|\rightarrow0}0, \end{equation} where the function $\varepsilon$ depends only on $K$ and $I$ and not otherwise on the mapping $\psi$.
\end{thm}


\section{Interpolation of Power Maps}
\label{interpolation_section}

This Section contains the primary technical contributions of the present work. We will describe in detail the interpolation procedure mentioned in the Introduction and prove the relevant estimates. We begin by describing the first step of the interpolation in logarithmic coordinates.

\begin{definition}\label{triangulation} We define a region \[W:=\{ z \in \mathbb{C} : 0<\textrm{Re}(z)<1 \} \setminus \{ z \in \mathbb{C} : 0\leq \textrm{Re}(z)\leq 1/2\textrm{ and }\textrm{Im}(z)\in2\mathbb{Z}+1 \}. \] Given $m\in\mathbb{N}$ with $m\geq2$, we also define a triangulation $T_{m}$ of $W$ as follows (see Figure \ref{fig:square.pdf_tex}). Place vertices at \begin{equation} 0\textrm{, } \left( 1+\frac{j}{m}\cdot i\right)_{j=0}^{m}\textrm{, and } \left( i + \frac{j}{2m-2}\right)_{j=0}^{m-1}.  \end{equation} Label the vertices black or white as follows: $0$ and $1$ are black, $i$ is white, and the other vertices are colored so that adjacent vertices on $\textrm{Re}(z)=1$ or $(\textrm{Im}(z)=1)\cap(\textrm{Re}(z)\leq1/2)$ have different colors. There is a triangulation of $W\cap[0,1]^2$ formed by connecting each vertex with real part $\leq1/2$ to $1+i\cdot(m-1)/m$. Iteratively reflecting this triangulation of $W\cap[0,1]^2$ along a subset of horizontal lines $\textrm{Im}(z)=k$, $k\in\mathbb{Z}$ defines the triangulation $T_m$.


\end{definition}

\begin{rem}\label{coloring} The coloring of the vertices in Definition \ref{triangulation} is not essential, however it is a useful convention. The colored vertices represent preimages of $\pm1$ under a map defined below in Definition \ref{quasiregular_interpolation}, with the two colors corresponding to the two choices $\pm1$.
\end{rem}

\begin{figure}
\centering
\scalebox{.5}{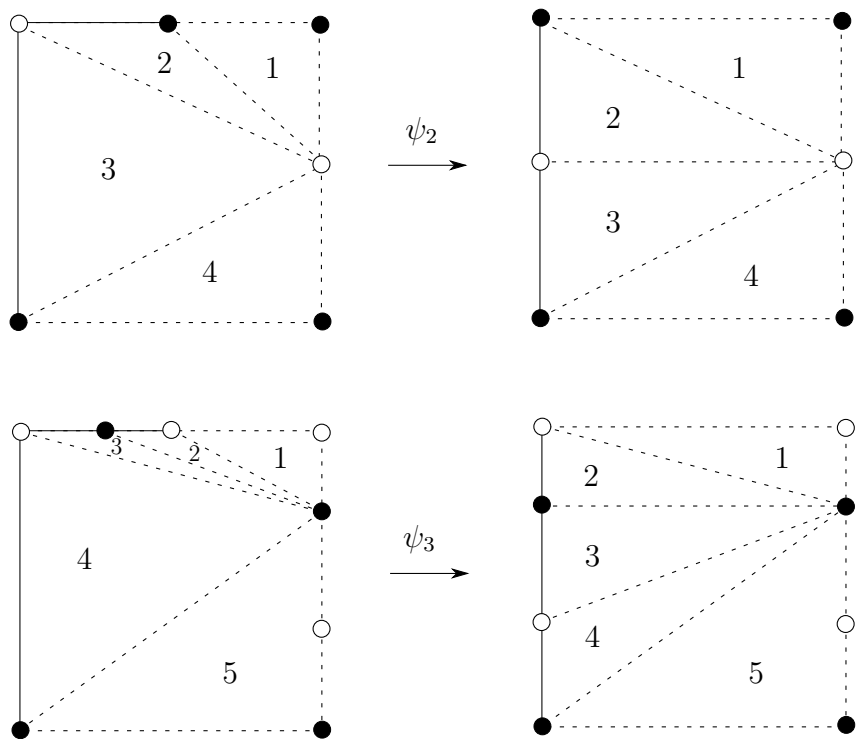}
\caption{ Illustrated is the definition of $\psi_m$ in $W\cap[0,1]^2$ in the cases $m=2$, $3$. }
\label{fig:square.pdf_tex}      
\end{figure}

\begin{definition}\label{triangulationT'} Given $m\in\mathbb{N}$, we define a triangulation $T_m'$ of $\{ z \in \mathbb{C} : 0<\textrm{Re}(z)<1 \}$ as follows (see Figure \ref{fig:square.pdf_tex}). First place vertices at \begin{equation} \left( \frac{j}{m}\cdot i\right)_{j=0}^{m} \textrm{ and } \left( 1+\frac{j}{m}\cdot i\right)_{j=0}^{m}. \end{equation} We color $0$ and $1$ black, and require adjacent vertices with the same real part to have different colors. There is a triangulation of $[0,1]^2$ defined by connecting each vertex on $\textrm{Re}(z)=0$ to $1+i\cdot(m-1)/m$. Iteratively reflecting this triangulation of $[0,1]^2$ along a subset of the horizontal lines $\textrm{Im}(z)=k$, $k\in\mathbb{Z}$ defines the triangulation $T_m'$.

\end{definition} 

\begin{rem} We remark that there is some flexibility in choosing the triangulations $T_m$, $T_m'$. For instance, one may instead define the triangulation $T_m'$ by connecting all of the vertices on $\textrm{Re}(z)=0$ to some other non-corner vertex on $\textrm{Re}(z)=1$ (instead of $1+i\cdot(m-1)/m$) without fundamentally altering the arguments that follow. 
\end{rem}

\noindent The triangulations $T_m$, $T_m'$ in Definitions \ref{triangulation} and \ref{triangulationT'} are \emph{compatible} in the following sense (see Section 4 of \cite{MR3421995}):

\begin{definition} Triangulations of two polygonal domains $\Omega_1$, $\Omega_2$ are compatible if we have a $1$-to-$1$
correspondence between the triangulations that preserves interior adjacencies (i.e., if two triangles share in edge in $\Omega_1$ then the corresponding triangles share an edge in $\Omega_2$).
\end{definition}


\begin{definition}\label{PWL_map_def} We define a PWL-homeomorphism \begin{equation} \psi_m: W \rightarrow \{ z \in \mathbb{C} : 0<\textrm{Re}(z)<1 \} \end{equation} as follows. We specify that \begin{align}\label{first_verter_definition} \psi_m(0):=0\textrm{ and } \psi_m( 1+\frac{j}{m}\cdot i ):=1+\frac{j}{m}\cdot i \textrm{ for } 0\leq j \leq m \textrm{, and } \\ \label{second_verter_definition} \psi_m\left(i + \frac{j}{2m-2}\right) :=  \frac{j+1}{m}\cdot i \textrm{ for } 0 \leq j \leq m-1. \end{align} As the triangulations $T_m\cap[0,1]^2$ and $T_m'\cap[0,1]^2$ are compatible, (\ref{first_verter_definition}) and (\ref{second_verter_definition}) uniquely determine a triangulation-preserving PWL-homeomorphism  \begin{equation} \psi_m: W\cap[0,1]^2 \rightarrow [0,1]^2 \end{equation} (see Section 4 of \cite{MR3421995}). We extend the map $\psi_m$ to a PWL-homeomorphism $\psi_m: W \rightarrow \{ z \in \mathbb{C} : 0<\textrm{Re}(z)<1 \}$ by repeated applications of the Schwarz-reflection principle (see Section I.8.4 of \cite{MR0344463} for a formulation of the Schwarz-reflection principle for quasiconformal mappings).



\end{definition}






\begin{prop}\label{psi_is_qc} For any $m\in\mathbb{N}$, the map \begin{equation} \psi_m: W \rightarrow \{ z \in \mathbb{C} : 0<\emph{Re}(z)<1 \} \end{equation} of Definition \ref{PWL_map_def} is quasiconformal. 
\end{prop}

\begin{proof} The map \begin{equation}\label{restricted_map} \psi_m: W\cap[0,1]^2 \rightarrow [0,1]^2 \end{equation} is defined as a PWL map, and hence the dilatation constant of (\ref{restricted_map}) is the supremum of the dilatations of the $m+2$ $\mathbb{R}$-linear maps in the definition of (\ref{restricted_map}). Thus (\ref{restricted_map}) is quasiconformal. As the definition of $\psi_m$ in $W$ is obtained by the Schwarz-reflection principle, it follows that $\psi_m$ is also quasiconformal with the same constant.  \end{proof}


\begin{rem}  $K(\psi_m)\rightarrow\infty$ as $m\rightarrow\infty$. 
\end{rem}


\noindent The above essentially defines the first step in our interpolation in logarithmic coordinates. We now revert back to the $z$-plane (see also Figure \ref{fig:adjustment_of_defn.pdf_tex}).

\begin{definition}\label{E_n} We define a planar region \[ E_n:=\exp\left(\frac{\pi}{n} \cdot W  \right). \] Given $n$, $m\in\mathbb{N}$, we also define a map \begin{equation}\label{definition_of_eta}\eta_{n, m}(z) := \left(z\mapsto \exp\left(\frac{\pi}{n}\cdot z\right)\right)\circ\psi_{m}\circ \left(z\mapsto \frac{n}{\pi}\log z\right) \textrm{ for } z\in E_{n}.  \end{equation}
\end{definition}

\begin{prop}\label{eta_dilatation_bound} For any $n, m\in\mathbb{N}$, the map \begin{equation} \eta_{n,m}: E_n \rightarrow \left\{ z \in \mathbb{C} : 1 \leq |z| \leq \exp(\pi/n) \right\}\end{equation} is a quasiconformal homeomorphism. Moreover, $K(\eta_{n,m})$ depends only on $m$.
\end{prop}

\begin{proof} This follows from (\ref{definition_of_eta}) and Proposition \ref{psi_is_qc}. \end{proof}



A candidate for an interpolation between $z\mapsto z^n$ on $|z|=1$ and $z\mapsto z^{mn}$ on $|z|=\exp(\pi/n)$ is the mapping $(z\mapsto z^{nm})\circ \eta_{n,m}$. Indeed, as we will shortly see, the map $(z\mapsto z^{nm})\circ \eta_{n,m}$ has the correct values on $|z|=1$ and $|z|=\exp(\pi/n)$. However, $(z\mapsto z^{nm})\circ \eta_{n,m}$ does not extend to a single-valued function across the radial arcs on the boundary of $E_n$. We will remedy this in Definition \ref{quasiregular_interpolation} below, for which we will first need the following:

\begin{definition}\label{sigma_definition} Following \cite{MR4023391}, we define a $3$-quasiconformal map \begin{equation}\sigma: \{ z \in \mathbb{C} : |z| > 1 \} \rightarrow \mathbb{C}\setminus[-1,1]\end{equation} as follows (see Figure \ref{fig:sigma_def}). Denote by $\mu(z):=(z+1)/(z-1)$ the M\"obius transformation mapping $|z|>1$ conformally to the right-half plane. Let $\nu$ denote the $3$-quasiconformal map sending the right-half plane to $\mathbb{C}\setminus(-\infty,0]$ that is the identity on $|\arg(z)|\leq \pi/4$ and triples angles in the remaining sector. Then \begin{equation} \sigma:=\mu\circ\nu\circ\mu. \end{equation}
\end{definition}

\begin{figure}[ht!]
{\includegraphics[width=1\textwidth]{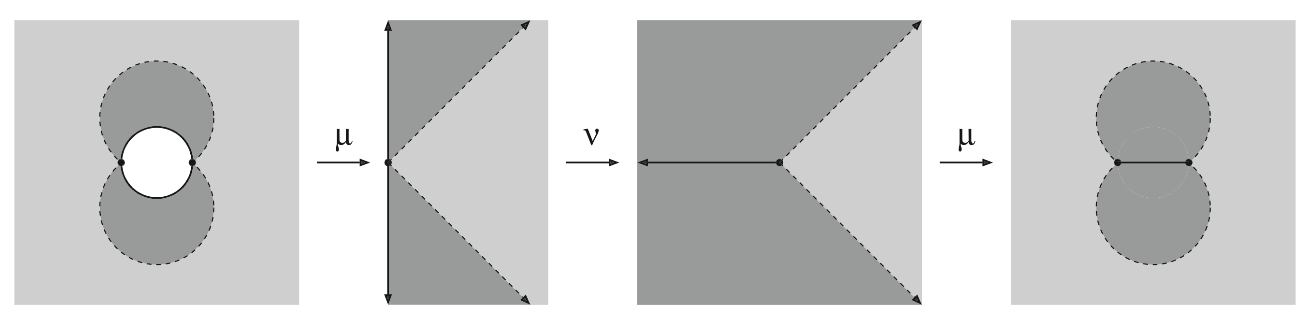}}
\caption{Illustrated is the map $\sigma$ of Definition \ref{sigma_definition}. This Figure is taken from \cite{MR4023391}.}
\label{fig:sigma_def}
\end{figure}

\begin{rem}\label{r_remark} We will denote $X:=\textrm{supp}(\sigma_{\overline{z}}/\sigma_{z})$ (illustrated as the dark gray region in the left-most copy of $\mathbb{C}$ in Figure \ref{fig:sigma_def}). Note that $\sigma$ is the identity on $ \{ z \in \mathbb{C} : |z| > 1 \}\setminus X$. 
\end{rem}

\begin{definition}\label{quasiregular_interpolation} Given $n$, $m\in\mathbb{N}$, we define a quasiregular map $g_{n,m}$ on the region $E_{n}$ as follows. Consider first the map \begin{equation}\label{interested_in_preimage} (z\mapsto z^{mn}) \circ \eta_{n,m}: E_{n} \rightarrow \left\{ z \in \mathbb{C} : 1 \leq |z| \leq \exp(m\pi) \right\}.\end{equation} Note that \[X\subset \left\{ z \in \mathbb{C} : 1 \leq |z| \leq \exp(m\pi) \right\} \textrm{ for any } m\in\mathbb{N}.\] There is a partition of the boundary of $E_{n}$ into the two circular arcs $|z|=1$, $|z|=\exp(\pi/n)$ and $n$ radial arcs perpendicular to $|z|=1$ (see Figure \ref{fig:adjustment_of_defn.pdf_tex} for the case $n=2$ where the radial arcs are colored red). The preimage of $X$ under (\ref{interested_in_preimage}) consists of $2mn$ components: let $U$ denote the union of those components which neighbor one of the $n$ radial arcs on the boundary of $E_n$. We define \begin{equation}\label{definition_of_g} g_{n,m}(z):= \begin{cases} 
      \sigma \circ (z\mapsto z^{mn}) \circ \eta_{n,m}(z) & z\in U \\
      (z\mapsto z^{mn}) \circ \eta_{n,m}(z) & z\in E_{n}\setminus U.
   \end{cases} \end{equation}
\end{definition}

\begin{figure}
\centering
\scalebox{.375}{
\begingroup%
  \makeatletter%
  \providecommand\color[2][]{%
    \errmessage{(Inkscape) Color is used for the text in Inkscape, but the package 'color.sty' is not loaded}%
    \renewcommand\color[2][]{}%
  }%
  \providecommand\transparent[1]{%
    \errmessage{(Inkscape) Transparency is used (non-zero) for the text in Inkscape, but the package 'transparent.sty' is not loaded}%
    \renewcommand\transparent[1]{}%
  }%
  \providecommand\rotatebox[2]{#2}%
  \newcommand*\fsize{\dimexpr\f@size pt\relax}%
  \newcommand*\lineheight[1]{\fontsize{\fsize}{#1\fsize}\selectfont}%
  \ifx\svgwidth\undefined%
    \setlength{\unitlength}{1040.79126553bp}%
    \ifx\svgscale\undefined%
      \relax%
    \else%
      \setlength{\unitlength}{\unitlength * \real{\svgscale}}%
    \fi%
  \else%
    \setlength{\unitlength}{\svgwidth}%
  \fi%
  \global\let\svgwidth\undefined%
  \global\let\svgscale\undefined%
  \makeatother%
  \begin{picture}(1,0.36188181)%
    \lineheight{1}%
    \setlength\tabcolsep{0pt}%
    \put(0,0){\includegraphics[width=\unitlength,page=1]{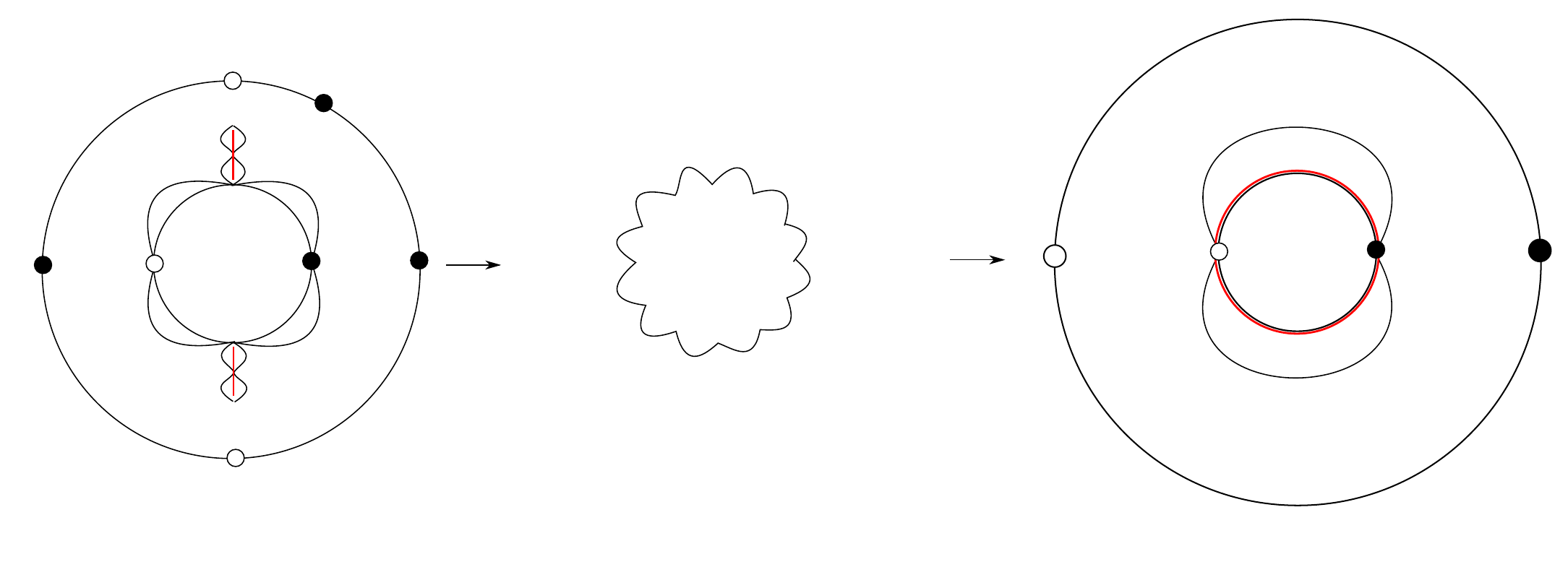}}%
    \put(0.28339587,0.21975778){\color[rgb]{0,0,0}\makebox(0,0)[lt]{\lineheight{1.25}\smash{\begin{tabular}[t]{l}\scalebox{2.5}{$\eta_{2,3}$}\end{tabular}}}}%
    \put(0,0){\includegraphics[width=\unitlength,page=2]{adjustment_of_defn.pdf}}%
    \put(0.5858391,0.21858806){\color[rgb]{0,0,0}\makebox(0,0)[lt]{\lineheight{1.25}\smash{\begin{tabular}[t]{l}\scalebox{2.5}{$z\mapsto z^6$}\end{tabular}}}}%
    \put(0.12740931,0.02436402){\color[rgb]{0,0,0}\makebox(0,0)[lt]{\lineheight{1.25}\smash{\begin{tabular}[t]{l}\scalebox{2.5}{(A)}\end{tabular}}}}%
    \put(0.44315197,0.02566882){\color[rgb]{0,0,0}\makebox(0,0)[lt]{\lineheight{1.25}\smash{\begin{tabular}[t]{l}\scalebox{2.5}{(B)}\end{tabular}}}}%
    \put(0.81038961,0.01140468){\color[rgb]{0,0,0}\makebox(0,0)[lt]{\lineheight{1.25}\smash{\begin{tabular}[t]{l}\scalebox{2.5}{(C)}\end{tabular}}}}%
  \end{picture}%
\endgroup%
}
\caption{The two dark-gray regions in (C) represent the set $X$ of Remark \ref{r_remark}, and the dark-gray regions in (B) and (A) represent the pullback of $X$ under the maps $z\mapsto z^6$ and $\eta_{2,3}\circ(z\mapsto z^6)$, respectively. The definition of $g_{2,3}$ differs from that of $(z\mapsto z^6)\circ\eta_{2,3}$ only in the four dark-gray regions in (A) neighboring the radial arcs colored red. The map $(z\mapsto z^6)\circ\eta_{2,3}$ is $2$-valued on these arcs, whereas $g_{2,3}$ is single-valued and indeed continuous across these arcs.}
\label{fig:adjustment_of_defn.pdf_tex}      
\end{figure}

\begin{rem} Note that the two formulas in (\ref{definition_of_g}) agree on $\partial U$ as $\sigma(z)=z$ for $z\in\partial X$.
\end{rem}

\begin{prop}\label{unit_interpolation_lemma} For any $n$, $m\in\mathbb{N}$, The map $g_{n,m}$ of Definition \ref{quasiregular_interpolation} is quasiregular on $1 \leq |z| \leq \exp(\pi/n)$. Moreover, $K(g_{n,m})$ depends only on $m$, and: \begin{enumerate} \item $g_{n,m}(z)=z^n$ if $|z|=1$, and \item $g_{n,m}(z)=z^{mn}$ if $|z|=\exp(\pi/n)$. \end{enumerate}
\end{prop}

\begin{proof} It is evident from Definition \ref{quasiregular_interpolation} that $g_{n,m}$ is quasiregular in $E_n$ as it is a composition of quasiregular maps in $E_n$. Thus to show that $g_{n,m}$ is quasiregular on $1 \leq |z| \leq \exp(\pi/n)$, it suffices (by a standard removability result) to show that $g_{n,m}$ extends continuously across the radial arcs on the boundary of $E_n$. For $z$ on such a radial arc, $(z\mapsto z^{mn}) \circ \eta_{n,m}(z)$ is $2$-valued with both values lying at complex-conjugate points on $\mathbb{T}$, whence $\sigma$ identifies these two points. Thus $g_{n,m}$ extends continuously across the radial arcs on the boundary of $E_n$. That $K(g_{n,m})$ depends only on $m$ follows from the Formula (\ref{definition_of_g}) and Proposition \ref{eta_dilatation_bound}.

It remains to show (1) and (2). (2) is evident since $|z|=\exp(\pi/n)$ is disjoint from $U$, and $\eta_{n,m}(z)=z$ for $|z|= \exp(\pi/n)$ (see Definition \ref{PWL_map_def} of $\psi_m$). We now verify (1). Consider the arc \[\mathcal{A}_n:=\{z : |z| = 1, 0\leq\textrm{arg}(z)\leq \pi/n\}.\] Since $\psi_m(0)=0$ and $\psi_m(i)=i/m$, it follows from (\ref{definition_of_eta}) that $\eta_{n,m}(\mathcal{A}_n)=\mathcal{A}_{mn}$. Thus $g_{n,m}(\mathcal{A}_n)=\mathcal{A}_1$ since $U$ is disjoint from $\mathcal{A}_n$. Thus $z\mapsto g_{n,m}(z)$ and $z\mapsto z^{n}$ agree set-wise on $\mathcal{A}_n$ and on the endpoints of $\mathcal{A}_n$. Since $z\mapsto z^n$ has derivative of constant modulus on $|z|=1$, we will be able to conclude that  $g_{n,m}(z)=z^n$ on $\mathcal{A}_n$ once we observe that $z\mapsto g_{n,m}(z)$ also has derivative of constant modulus on $|z|=1$. Indeed, it follows from the chain rule and formula (\ref{definition_of_eta}) that $\eta_{n,m}$ has derivative of constant modulus on $\mathcal{A}_n$, whence so does $g_{n,m}$ since $g_{n,m}(z)=(z\mapsto z^{mn}) \circ \eta_{n,m}(z)$ on $\mathcal{A}_n$. The argument that $z\mapsto g_{n,m}(z)$ and $z\mapsto z^n$ agree on the other subarcs \[ \left\{z : |z| = 1, \frac{(k-1)\pi}{n}\leq\textrm{arg}(z)\leq \frac{k\pi}{n}\right\}, \hspace{3mm} 1 \leq k \leq n+1 \] of $|z|=1$ is similar. \end{proof}


In order to understand the singular values of the function $f$ of Theorem \ref{mainthm}, we will need to keep track of those points at which our interpolating function is locally $n:1$ for $n>1$. To this end, we introduce the following definition:

\begin{definition} Let $g$ be a quasiregular function, defined in a neighborhood of a point $z\in\mathbb{C}$. We say that $z$ is a \emph{branched point} of $g$ if for any neighborhood $U$ of $z$, the map $g|_U$ is $n:1$ onto its image for $n>1$. If, further, $n=2$, we say that $z$ is a \emph{simple} branched point. We say $w\in\mathbb{C}$ is a \emph{branched value} of $g$ if $w=g(z)$ for a branched point $z$ of $g$. 
\end{definition}

\begin{prop}\label{simplest_critical_points_location} Let $g$ be as in Definition \ref{quasiregular_interpolation}. Define $g(z):=z^n$ for $|z|<1$, and $g(z):=z^{mn}$ for $|z|>\exp(\pi/n)$. Then the only non-zero branched points of $g$ are: \begin{equation}\label{crit_pts} \exp\left(  \frac{\pi}{n}\cdot\frac{ l}{(2m-2)}+i\cdot\frac{\pi(2k-1)}{n} \right)\hspace{1mm}\textrm{   for   } \hspace{1mm}0\leq l \leq m-2\textrm{, and } 1\leq k \leq n. \end{equation} Moreover, all non-zero branched points of $g$ are simple, and the only non-zero branched values of $g$ are $\pm1$. 
\end{prop}

\begin{rem} The expression (\ref{crit_pts}) is simply a formula for the black and white vertices located on the radial arcs colored red in, for instance, Figure \ref{fig:adjustment_of_defn.pdf_tex}(A), but excluding those vertices at the outer-most tips of the radial arcs. 
\end{rem}

\begin{proof} We first note that the extended formula defines a quasiregular function on $\mathbb{C}$ by removability of analytic arcs for quasiregular mappings. It is readily verified then from the definition of $g$ that any neighborhood of each of the points (\ref{crit_pts}) is mapped $2:1$ onto its image by $g$. The points (\ref{crit_pts}) are sent to $\pm1$ by $g$, where we note that the sign may be determined by the coloring of the vertex as described in Remark \ref{coloring}. Lastly, again from the definition of $g$, one verifies directly that there are no remaining non-zero branched points of $g$. \end{proof}


 We will also need to record the zeros of our interpolating function for later application. The formulas are listed below in Proposition \ref{location_of_zeros}, but they are readily seen to just be the midpoint between each adjacent black/white vertex on a radial segment pictured in, for instance, Figure \ref{fig:adjustment_of_defn.pdf_tex}(A).

\begin{prop}\label{location_of_zeros} Let $g$ be as in Proposition \ref{simplest_critical_points_location}. Then the zeros of $g$ are given by: \begin{align}\label{listing_of_zeros} 0\textrm{ and }\exp \left( \frac{\pi}{n} \cdot \frac{l+\frac{1}{2}}{2m-2} + i \cdot \frac{\pi(2k-1)}{n} \right) \hspace{1mm}\textrm{   for   } \hspace{1mm}0\leq l \leq m-2\textrm{ and } 1\leq k \leq n, \end{align} all of which are simple except for $0$ which is of multiplicity $n$.
\end{prop}

\begin{proof} Recalling Definition \ref{sigma_definition} of $\sigma$, we see that $\sigma^{-1}(0)=\{\pm i\}$. Let $S:=(z\mapsto z^{mn})^{-1}(\pm i)$. Then $g^{-1}(0)$ is readily seen to be those points $x$ such that $\eta_{n,m}(x)\in S$ \emph{and} $x\in U$ for $U$ as in Definition \ref{quasiregular_interpolation}. Such $x$ are listed in (\ref{listing_of_zeros}). \end{proof}



 We now generalize our interpolating function slightly to allow for interpolation between $z\mapsto z^n$ and $z\mapsto z^M$ for $M>n$, where $M$ is not necessarily a multiple of $n$. While there is a necessary level of complication in the formulas in the following proof, the idea is simple. Denoting $m:=\lfloor M/n \rfloor$, we will use $\eta_{n,m}$ in some parts of the interpolating annulus, and $\eta_{n,m+1}$ in others, so that we add the necessary amount $M-n$ of new vertices (branched points). Next, we adjust by a homeomorphism $\tau$ of the circle which sends the $n$ old vertices and $M-n$ new vertices on $\mathbb{T}$ together to span the $M^{\textrm{th}}$ roots of unity, and finally post-compose with an adjustment of the map $z\mapsto z^M$ as in Definition \ref{quasiregular_interpolation}.

\begin{prop}\label{unit_interpolation_lemma_general} Let $n$, $M\in\mathbb{N}$ with $M>n$. Then there exists a quasiregular function \begin{equation} g: \{ z \in \mathbb{C} : 1 \leq |z| \leq \exp(\pi/n) \} \rightarrow \{ z \in \mathbb{C} : |z| \leq \exp(M\pi/n) \}, \end{equation} such that: \begin{enumerate} \item $g(z)=z^n$ if $|z|=1$, \item $g(z)=z^{M}$ if $|z|=\exp(\pi/n)$, and \item $K(g)$ depends only on $M/n$, and not otherwise on $M$ or $n$. \end{enumerate}
\end{prop}

\begin{proof} Consider the region $E_n$ as in Definition \ref{E_n}. Let $m:=\lfloor M/n \rfloor$. Let \begin{equation} E_{n,j}:= E_n \cap \left\{ z\in\mathbb{C}\setminus\{0\} : \frac{2(j-1)\pi}{ n} \leq \arg(z) \leq \frac{2j\pi }{n} \right\} \textrm{ for } 1\leq j \leq n.\end{equation} Let $p:=n-M+nm$, and recall the mapping $\eta_{n,m}$ as in Definition \ref{E_n}, where we use the convention $\eta_{n,0}(z)\equiv z$. We define the mapping \begin{equation}\label{definition_of_eta_general} \eta(z):= \begin{cases} 
      \eta_{n,m}(z) & z\in \cup_{j=1}^{p}E_{n,j} \\
      \eta_{n,m+1}(z) & z\in \cup_{j=p+1}^{n}E_{n,j}.
   \end{cases} \end{equation} It is readily verified that $\eta: E_n \rightarrow \{ z \in \mathbb{C} : 1 \leq |z| \leq \exp(\pi/n)\}$ is a homeomorphism, since $\psi_m$, $\psi_{m+1}$ are both the identity on the lines $y=2\mathbb{Z}$. Next, we define a homeomorphism $\tau: \mathbb{T} \rightarrow \mathbb{T}$ by \begin{equation}\label{definition_of_tau} \tau(\theta):= \begin{cases} 
       \frac{nm}{M}\cdot\theta & \theta \in [0,2p\pi/n] \\
       \frac{(m+1)n}{M}\cdot\theta & \theta \in [-2\pi(n-p)/n,0].
   \end{cases} \end{equation} Extend $\tau$ to a self-homeomorphism of $A(1,\exp(\pi/n))$ by linearly interpolating $\tau|_{\mathbb{T}}$ with the identity on $|z|=\exp(\pi/n)$. Since $p/n$ and $(n-p)/n$ depend only on $M/n$, it is readily seen from (\ref{definition_of_tau}) that $K(\tau)$ depends only on $M/n$, and not otherwise on $M$ or $n$.
   
Finally, we define $g$ by appropriately adjusting Definition \ref{quasiregular_interpolation}. Namely, let $X:=\textrm{supp}(\sigma_{\overline{z}}/\sigma_{z})$ as in Remark \ref{r_remark}. Then the pullback of $X$ under \begin{equation} (z\mapsto z^M)\circ\tau\circ\eta: E_n \rightarrow A(1, \exp(M\pi/n)) \end{equation} consists of $2M$ components. Denoting by $U$ the union of those $2(M-n)$ components which neighbor a radial arc on the boundary of $E_n$, we again define \begin{equation}\label{definition_of_g_general} g(z):= \begin{cases} 
      \sigma \circ (z\mapsto z^{M}) \circ \eta(z) & z\in U \\
      (z\mapsto z^{M}) \circ \eta(z) & z\in E_{n}\setminus U.
   \end{cases} \end{equation} The proof that $g$ satisfies (1)-(3) then is the same as in Proposition \ref{unit_interpolation_lemma}. \end{proof}


Following Proposition \ref{simplest_critical_points_location}, we will also list the branched points, branched values, and zeros of the map $g$ of Proposition \ref{unit_interpolation_lemma_general}. But first, we describe a simple rescaling which allows our annular region of interpolation to have an inner boundary lying on a circle of variable radius.

\begin{prop}\label{interpolation_with_x} Let $r>0$, $c\in\mathbb{C}^{\star}$ and $n$, $M\in\mathbb{N}$ with $M>n$. There exists a quasiregular function \begin{equation} g: \{ z \in \mathbb{C} : r \leq |z| \leq r\exp(\pi/n) \} \rightarrow \{ z \in \mathbb{C} : |z| \leq cr^{n}\exp(M\pi/n) \}, \end{equation} such that: \begin{enumerate} \item $g(z)=c\cdot z^n$ if $|z|=r$, \item $g(z)=c\cdot z^{M}/r^{M-n}$ if $|z|=r\exp(\pi/n)$, and \item $K(g)$ depends only on $M/n$, and not otherwise on $M$, $n$, $r$ or $c$. \end{enumerate}
\end{prop}

\begin{proof} We first consider the case when $M=m\cdot n$ for $m\in\mathbb{N}$. Let \begin{equation} E_{n,r} := \exp\left(\frac{\pi}{n}\cdot W +\log r\right).\end{equation} By Proposition \ref{eta_dilatation_bound}, the map \begin{equation} \eta_{n,m,r}:=\eta_{n,m} \circ (z\mapsto z/r): E_{n,r} \rightarrow \left\{ z \in \mathbb{C} : 1 \leq |z| \leq \exp(\pi/n) \right\}   \end{equation} is a quasiconformal homeomorphism where $K(\eta_{n,m,r})$ depends only on $m$. Following Definition \ref{quasiregular_interpolation}, we can define a map $\tilde{g}$ in $E_{n,r}$ by adjusting the values of $(z\mapsto z^{mn}) \circ \eta_{n,m,r}$ in a neighborhood of the radial arcs on the boundary of $E_{n,r}$. The proof that $\tilde{g}$ is quasiregular on $\{ z \in \mathbb{C} : r \leq |z| \leq r\exp(\pi/n) \}$ is then the same as in Proposition \ref{unit_interpolation_lemma}. Set \begin{align} g(z):=c\cdot r^n\cdot\tilde{g}(z). \end{align} Then the proof that $g$ satisfies (1)-(3) similarly follows as in Proposition \ref{unit_interpolation_lemma}. Lastly, the case when $M>n$ is not necessarily a multiple of $n$ follows from the same adjustment of the above interpolation as in Proposition \ref{unit_interpolation_lemma_general}. \end{proof}

\begin{rem} When we wish to emphasize the dependence of $g$ on the parameters $n$, $M$, $r$, $c$, we will write $g_{n, M, r, c}$ (in that order).
\end{rem}

\noindent Lastly, we record the branched points, branched values, and zeros of our interpolating function. 

\begin{prop}\label{crit_pts_more_comp} Let $g$, and notation be as in Proposition \ref{interpolation_with_x}. Set $m:=\lfloor M/n \rfloor$, and $p:=n-M+nm$. Define $g(z):=cz^n$ for $|z|<r$, and $g(z):=cz^{M}/r^{M-n}$ for $|z|>r\exp(\pi/n)$. Then the only non-zero branched points of $g$ are: \begin{align}\label{crit_pts_comp} r\cdot\exp\left(  \frac{\pi}{n}\cdot\frac{ l}{(2m-2)}+i\cdot\frac{\pi(2k-1)}{n} \right)\hspace{1mm}\textrm{   for   } \hspace{1mm}0\leq l \leq m-2\textrm{ and } 1\leq k \leq p\textrm{, and} \\  r\cdot\exp\left(  \frac{\pi}{n}\cdot\frac{ l}{2m}+i\cdot\frac{\pi(2k-1)}{n} \right)\hspace{1mm}\textrm{   for   } \hspace{1mm}0\leq l \leq m-1\textrm{, and } p+1\leq k \leq n. \nonumber\phantom{asdf}\end{align} Moreover, all non-zero branched points of $g$ are simple, and the only non-zero branched values of $g$ are $\pm cr^n$. The zeros of $g$ are given by the multiplicity $n$ zero at $0$, and the simple zeros described by replacing $l$ with $l+1/2$ in the expressions (\ref{crit_pts_comp}).
\end{prop}

\begin{proof} The proof is essentially the same as the proofs of Propositions \ref{simplest_critical_points_location} and \ref{location_of_zeros}. We provide a sketch. By the definition of $g$, a small neighborhood of any point in (\ref{crit_pts_comp}) is mapped $2:1$ onto its image, and each of the points in (\ref{crit_pts_comp}) is mapped to $\pm cr^n$. Thus each of the points in (\ref{crit_pts_comp}) is a simple branched point, with corresponding branched values $\pm cr^n$.  A small neighborhood of any other non-zero point (not listed in (\ref{crit_pts_comp})) is mapped homeomorphically onto its image, thus (\ref{crit_pts_comp}) lists all the branched values of $g$. Similarly, the statement about the zeros of $g$ follows by inspection of which points $z\mapsto(z\mapsto z^M)\circ \eta(z)$ maps to $\pm i$ as in the proof of Proposition \ref{location_of_zeros}. \end{proof}



\section{Entire Functions}
\label{Entire_function_section}

With the technical work of Section \ref{interpolation_section} behind us, we will now apply our interpolation repeatedly between power maps of increasing degree in ``increasing'' annuli in the plane. As long as the degree of the power maps increases by at most a fixed constant factor in each consecutive annulus, this procedure gives a function quasiregular in $\mathbb{C}$. We formalize this below.


\begin{definition}\label{general_powers_def} Let $c\in\mathbb{C}^\star$, $(M_j)_{j=1}^\infty \in \mathbb{N}$ be increasing, and $(r_j)_{j=1}^\infty \in \mathbb{R}^+$. We set  $r_0:=0$,  $M_0:=1$. Suppose that \begin{equation}\label{growth_condition} r_{j+1} \geq  \exp\left(\pi\big/M_j\right) \cdot r_j \textrm{ for all } j\in\mathbb{N}, \textrm{ and } r_j\xrightarrow{j\rightarrow\infty}\infty. \end{equation} Set \begin{equation}\label{c_defn} c_1:=c\textrm{, and } c_j:=c_{j-1}\cdot r_{j-1}^{M_{j-1}-M_{j}}=c\cdot\prod_{k=2}^{j}r_{k-1}^{M_{k-1}-M_{k}} \textrm{ for } j\geq2. \end{equation}  We then define: \begin{equation}\label{h_formula} h(z):= \begin{cases} 
      c_j\cdot z^{M_j} & \textrm{ if } \hspace{2mm} r_{j-1}\cdot\exp(\pi/M_{j-1}) \leq |z| \leq r_j \\
     g_{M_{j}, M_{j+1}, r_j, c_j}(z) & \textrm{ if } \hspace{2mm}  r_j \leq |z| \leq r_j\cdot \exp(\pi/M_j). 
   \end{cases} \end{equation} over all $j\in\mathbb{N}$. \end{definition}
   
 \begin{rem} By Proposition \ref{interpolation_with_x}, the two definitions in (\ref{h_formula}) agree on $|z|=r_j$. By Proposition \ref{interpolation_with_x} and the definition (\ref{c_defn}) of $c_j$, the two definitions in (\ref{h_formula}) agree on $|z|=r_j\cdot\exp(\pi/M_j)$. Thus the formula (\ref{h_formula}) determines a well-defined function $h: \mathbb{C} \rightarrow \mathbb{C}$.


 \end{rem}

\begin{definition} We will say that $(M_j)_{j=1}^\infty \in \mathbb{N}$, $(r_j)_{j=1}^\infty \in \mathbb{R}^+$ are \emph{weakly permissible} if (\ref{growth_condition}) is satisfied.
\end{definition}

\begin{rem} Suppose $(M_j)_{j=1}^\infty$, $(r_j)_{j=1}^\infty$ are weakly permissible. Then, according to Definition \ref{permissible}, we have that $(M_j)_{j=1}^\infty$, $(r_j)_{j=1}^\infty$ are permissible if and only if the sequence $(M_j/M_{j-1})_{j=1}^\infty$ is bounded.

\end{rem}

\begin{rem} The map $h$ of Definition \ref{general_powers_def} is determined by a choice of $c\in\mathbb{C}^\star$ and weakly permissible $(M_j)_{j=1}^\infty$, $(r_j)_{j=1}^\infty $.  
\end{rem}

\begin{prop}\label{quasiregularity} Suppose $(M_j)_{j=1}^\infty$, $(r_j)_{j=1}^\infty$ are weakly permissible, and $c\in\mathbb{C}^\star$. Then the function $h$ as defined in Definition \ref{general_powers_def} is quasiregular on compact subsets of $\mathbb{C}$. If, moreover, $(M_j)_{j=1}^\infty$, $(r_j)_{j=1}^\infty$ are permissible, then $h$ is quasiregular on $\mathbb{C}$ and $K(h)$ depends only on $\emph{sup}_j(M_j/M_{j-1})$ and not otherwise on $(M_j)_{j=1}^\infty$, $(r_j)_{j=1}^\infty$.
\end{prop}

\begin{proof} In each annular region in the Definition (\ref{h_formula}), the map $h$ is either analytic, or quasiregular by Proposition \ref{interpolation_with_x}. By removability of analytic arcs, the map $h$ is therefore quasiregular across the boundaries of the annular regions. As any compact subset of $\mathbb{C}$ meets only finitely many such annular regions, it follows that $h$ is locally quasiregular. If the sequence $(M_j/M_{j-1})_{j=1}^\infty$ is bounded (or equivalently $(M_j)_{j=1}^\infty$, $(r_j)_{j=1}^\infty$ are permissible), then each of the maps $g_{M_j, M_{j+1}, r_j, c_j}(z)$ have a dilatation bounded uniformly over $j$ (with a bound depending only on $\textrm{sup}_j(M_j/M_{j-1})$) by (3) of Proposition \ref{interpolation_with_x}, and so the last statement of the Proposition follows. \end{proof}


The definition of \emph{permissible} thus ensures that the formula (\ref{h_formula}) defines a quasiregular function, in which case we may integrate the Beltrami coefficient $h_{\overline{z}}/h_z$ to obtain a quasiconformal map $\phi$ such that $h\circ\phi^{-1}$ is holomorphic. If, moreover, the $M_j$ increase sufficiently quickly (see (\ref{summability})), the total region of interpolation is sufficiently small to guarantee conformality of $\phi$ at $\infty$. This is the content of Theorem \ref{TWB_application} below, which will be presented after the following definition:

\begin{definition}\label{strongly_permissible} Let $(M_j)_{j=1}^\infty, (r_j)_{j=1}^\infty$ be permissible. We say that  $(M_j)_{j=1}^\infty, (r_j)_{j=1}^\infty$ are \emph{strongly permissible} if also \begin{equation}\label{summability} \sum_jM_j^{-1}<\infty. \end{equation}
\end{definition}

\begin{thm}\label{TWB_application} Let $(M_j)_{j=1}^\infty$, $(r_j)_{j=1}^\infty $ be strongly permissible, and $c\in\mathbb{C}^\star$. Then there exists a quasiconformal mapping $\phi: \mathbb{C} \rightarrow \mathbb{C}$ such that $f:=h\circ\phi^{-1}$ as in Theorem \ref{mainthm} is holomorphic, and \begin{equation}\label{conformality_at_infty} \left|\frac{\phi(z)}{z} - 1\right|\xrightarrow{z\rightarrow\infty} 0.  \end{equation}
\end{thm}

\begin{proof} The proof is an application of Theorem \ref{TWB} (the Teichm\"uller-Wittich-Belinskii Theorem). The existence of a quasiconformal $\phi: \mathbb{C} \rightarrow\mathbb{C}$ such that $h\circ\phi^{-1}$ is holomorphic and $\phi(0)=0$ follows from applying the Measurable Riemann Mapping Theorem to $h_{\overline{z}}/h_z$.  Let \begin{equation} \psi(z):=1/\phi(1/z). \end{equation} Then $\psi$ is a quasiconformal self-mapping of $\mathbb{C}$ satisfying $\psi(0)=0$. Let $K$ denote the quasiconformal constant of $\psi$, and take $r<\infty$. We calculate \begin{align}  I(r):=\frac{1}{2\pi}\int_{|z|<r} \frac{D(z)-1}{|z|^2}\emph{d}A(z) < \frac{1}{2\pi}\int_{|z|<1} \frac{K-1}{|z|^2}\emph{d}A(z) = \frac{1}{2\pi}\sum_{j=1}^\infty \int_{B_j}\frac{K-1}{|z|^2}\emph{d}A(z), \end{align} where \begin{equation} B_j:=\{z\in\mathbb{C} : r_j^{-1}\exp(-\pi/M_j) \leq |z| \leq r_j^{-1} \}. \end{equation} We continue our calculation: \begin{align}\label{calculation} \frac{1}{2\pi}\sum_{j=1}^\infty \int_{B_j}\frac{K-1}{|z|^2}\emph{d}A(z) \leq \frac{K-1}{2\pi}\sum_{j=1}^\infty \int_{B_j} \frac{1}{r_j^{-2}\exp(-2\pi/M_j)}\emph{d}A(z) = \\ \frac{K-1}{2\pi}\sum_{j=1}^\infty \frac{\pi(r_j^{-2}-r_j^{-2}\exp(-2\pi/M_j))}{r_j^{-2}\exp(-2\pi/M_j)} =  \frac{K-1}{2}\sum_{j=1}^\infty\big(\exp(2\pi/M_j)-1\big).  \nonumber \end{align} The infinite sum on the right-hand side of (\ref{calculation}) converges if and only if \begin{equation} \prod_{j=1}^\infty \exp(2\pi/M_j) \end{equation} converges, which is readily seen to be the case by assumption of (\ref{summability}). Thus, by Theorem \ref{TWB}, the limit \begin{equation} \psi_z(0):=\lim_{z\rightarrow0}\frac{\psi(z)}{z} \end{equation} exists, and we have \begin{equation} \left| \frac{\psi(z)}{z} - \psi_z(0) \right| \xrightarrow{|z|\rightarrow0} 0. \end{equation} By multiplying $\psi$ by $1/\psi_z(0)$ if necessary, we can further assume that $\psi_z(0)=1$. Since $\psi(z)=1/\phi(1/z)$, (\ref{conformality_at_infty}) follows. \end{proof}




\noindent Next we list the singularities of the entire function $f$:

\begin{prop}\label{critical_point_listing}
Let $(M_j)_{j=1}^\infty$, $(r_j)_{j=1}^\infty$ be permissible, $c\in\mathbb{C}^\star$, and $f$, $\phi$ as in Theorem \ref{TWB_application}. Set $m_j:=\lfloor M_j/M_{j-1} \rfloor$, and $p_j:=M_{j-1}-M_j+M_{j-1}m_j$. Then the only critical points of $f$ are $0$ and the simple critical points given by \begin{align}\label{final_crit_pts_listing} \phi\left( r_j\cdot \exp\left(\frac{\pi}{M_{j-1}} \cdot \frac{l_j}{2m_j-2}+i\frac{(2k_j-1)\pi}{M_{j-1}} \right) \right)\textrm{, and } \\ \phi\left(  r_j\cdot \exp\left(\frac{\pi}{M_{j-1}} \cdot \frac{l_j'}{2m_j}+i\frac{(2k_j'-1)\pi}{M_{j-1}} \right) \right),\phantom{, and }\nonumber\end{align} where $j\in\mathbb{N}$, and $1 \leq k_j \leq p_j$, $0\leq l_j \leq m_j-2$, and $p_j+1 \leq k_j' \leq M_{j-1}$, $0\leq l_j' \leq m_j-1$. The only singular values of $f$ are the critical values $(\pm c_jr_j^{M_j})_{j=0}^\infty$.

\end{prop}

\begin{proof} By Proposition \ref{crit_pts_more_comp}, the only branched points of $h$ are $0$ and the points given by (\ref{final_crit_pts_listing}) without the $\phi$ factor. Thus, as $f:=h\circ\phi^{-1}$, it follows the only critical points of $f$ are $\phi(0)=0$ and those given in (\ref{final_crit_pts_listing}). Again, as $f:=h\circ\phi^{-1}$, by Proposition \ref{crit_pts_more_comp} each of the points in (\ref{final_crit_pts_listing}) is mapped to $\pm c_jr_j^{M_j}$ as $j$ ranges over $\mathbb{N}$. There are no asymptotic values of $f$: if $\gamma\rightarrow\infty$ is a curve, then $f(\gamma)$ is unbounded.  \end{proof}

\begin{rem}\label{zero_listing} The same proof as for Proposition \ref{critical_point_listing} shows that the zeros of $f$ are given by $0$ (multiplicity $M_1$) and the simple zeros whose formulae is given by replacing $l_j$ and $l_j'$ in the expressions in (\ref{final_crit_pts_listing}) by $l_j+1/2$, $l_j'+1/2$, respectively.
\end{rem}

\noindent Theorem \ref{mainthm} now follows:

\vspace{3mm}

\noindent \emph{Proof of Theorem \ref{mainthm}:} We take $f:=h\circ\phi^{-1}$ as in Theorem \ref{TWB_application}. The conclusions of Theorem \ref{mainthm} are then included in the statements of Theorem \ref{TWB_application} and Proposition \ref{critical_point_listing}. 

\vspace{3mm}

The applications of the present manuscript of which the authors are aware are in regards to entire functions with an essential singularity at $\infty$, and so it is natural to impose the assumption $r_j\rightarrow\infty$ in Definition \ref{permissible} of ``permissibility''. However, this is inessential to most of our arguments. Indeed, if we assume instead that $r_j\rightarrow r_\infty<\infty$, we have the following:



\begin{thmx}\label{mainthm2} Let $c\in\mathbb{C}^\star$, $0<r_\infty<\infty$, $(M_j)_{j=1}^\infty \in \mathbb{N}$ be increasing, $(r_j)_{j=1}^\infty \in \mathbb{R}^+$ such that \begin{gather} r_{j+1} \geq  \exp\left(\pi\big/M_j\right) \cdot r_j \textrm{ for all } j\in\mathbb{N}\textrm{, } r_j\xrightarrow{j\rightarrow\infty}r_\infty\textrm{, and } \sup_j\frac{M_{j+1}}{M_j}<\infty.  \end{gather} Set $r_0:=0$ and
\begin{equation}  c_1:=c\textrm{, and } c_j:=c_{j-1}\cdot r_{j-1}^{M_{j-1}-M_{j}} \textrm{ for } j\geq2.   \end{equation}  Then there exists a holomorphic function $f: r_\infty\mathbb{D} \rightarrow\mathbb{C}$ and a $K$-quasiconformal homeomorphism $\phi: r_\infty\mathbb{D} \rightarrow r_\infty\mathbb{D}$ with $K$ depending only on $\sup_j(M_{j+1}/M_j)$ such that: \begin{gather}\label{theoremBfirstformula} f\circ\phi(z)=c_jz^{M_j} \textrm{ for } r_{j-1}\cdot\exp(\pi/M_{j-1}) \leq |z| \leq r_j\textrm{, } j\in\mathbb{N}. \end{gather} Moreover, the only singular values of $f$ are the critical values $(\pm c_jr_j^{M_j})_{j=1}^\infty$.
\end{thmx}

\begin{proof} We define $h$ exactly as in (\ref{h_formula}), so that the domain of $h$ is now $r_\infty\mathbb{D}$ (rather than $\mathbb{C}$ as in Theorem \ref{mainthm}). The same argument given in Proposition \ref{quasiregularity} proves that $h$ is $K$-quasiregular, with $K$ depending only on $\sup_j(M_{j+1}/M_j)$. By applying the Measurable Riemann Mapping theorem to the Beltrami coefficient $h_{\overline{z}}/h_z$ (defined on $r_\infty\mathbb{D}$), we have a $K$-quasiconformal homeomorphism $\phi: r_\infty\mathbb{D} \rightarrow r_\infty\mathbb{D}$ such that $f:=h\circ\phi^{-1}$ is holomorphic on $r_\infty\mathbb{D}$. The formula (\ref{theoremBfirstformula}) now follows from the definitions of $f$ and $h$, and the statement about the singular values of $f$ follows exactly as in the proof of Proposition \ref{critical_point_listing}. \end{proof}

Missing from the statement of Theorem \ref{mainthm2} is an application of criteria for conformality at a point (in Theorem \ref{mainthm} the criteria is $\sum_j M_j^{-1}<\infty$, and the point is $\infty$). Although we do not record them here, one can produce analogous criteria in Theorem \ref{mainthm2} for conformality of $\phi$ at points on $|z|=r_\infty$. Indeed, we may extend $\phi$ to a self-map of $\hat{\mathbb{C}}$ by the Schwarz reflection principle, and apply Theorem \ref{TWB} at a point lying on $|z|=r_\infty$.



\section{A Dynamical Application}
\label{dynamical_application_section}

In this Section, we briefly discuss some dynamical applications of Theorem \ref{mainthm}. We will discuss an approach to constructing entire functions with multiply-connected wandering domains using Theorem \ref{mainthm}. But first we will show how, in certain applications, we can conclude much more than conformality at $\infty$ of the map $\phi$ in Theorem \ref{mainthm}. Namely we can conclude a uniform estimate $||\phi(z)/z-1||_{L^\infty(\mathbb{C})}<\varepsilon$ in the following situation: given a function $f=h\circ\phi^{-1}$ as in Theorem \ref{mainthm}, if we define $h_n$ by replacing $h$ in $|z|\leq r_n$ with $z\mapsto c_nz^{M_n}$, we obtain a sequence of entire functions $f_n=h_n\circ\phi_n^{-1}$ with $f_n\approx f$ in $|z|\geq r_n$. As $n\rightarrow\infty$, the maps $\phi_n$ converge uniformly (in the spherical metric on $\hat{\mathbb{C}}$) to the identity so that $||\phi_n(z)/z-1||_{L^\infty(\mathbb{C})}<\varepsilon$ for large $n$. This argument is very useful in dynamical applications: it will be used in the companion paper \cite{BL_inprep}, and a related argument is used in \cite{MR3316755}, \cite{MR4041106}, \cite{2019arXiv190410086L}. We formalize the above discussion below.


\begin{definition}\label{entire_family} Let \begin{align}\label{choice_of_pars} (M_j)_{j=1}^\infty, (r_j)_{j=1}^\infty \textrm{ be strongly permissible, and } c\in\mathbb{C}^\star. \end{align} Denote by $f=h\circ\phi^{-1}$ the entire function obtained by applying Theorem \ref{mainthm} to (\ref{choice_of_pars}), where $h$ is quasiregular and $\phi$ quasiconformal. We define a sequence $(f_n)_{n=0}^\infty$ of entire functions as follows. Let $h_0:=h$ and for $n\geq1$:  \[ h_n(z)=\begin{cases} 
      c_nz^{M_n} & \textrm{ if } |z|\leq r_n \\
      h(z) & \textrm{ if } |z|\geq r_n.
   \end{cases}
\] Define $f_n:=h_n\circ\phi_n^{-1}$, where $\phi_n: \mathbb{C}\rightarrow\mathbb{C}$ is the unique quasiconformal mapping such that: \begin{enumerate} \item $f_n$ is holomorphic, \item $\phi_n(0)=0$, and \item $|\phi_n(z)/z - 1|\rightarrow0$ as $z\rightarrow\infty$.  \end{enumerate}
\end{definition}

\begin{rem}\label{normalization_remark} That we may normalize $\phi_n$ in Definition \ref{entire_family} so as to satisfy (3) requires justification. To this end, let \begin{equation}\label{I_n} I_n(r):=\frac{1}{2\pi}\int_{|z|<r} \frac{D_n(z)-1}{|z|^2}\textrm{d}A(z)\textrm{ for } r>0, \end{equation} where $D_n$ is the dilatation quotient (see Definition \ref{dilatation_definition}) of $\psi_n(z):=1/\phi_n(1/z)$. Then \begin{equation}\label{D_n_relation} D_n(z)=0 \textrm{ if } |z|^{-1} \leq r_n\textrm{, and } D_n(z)\stackrel{a.e.}{=} D_0(z) \textrm{ if } |z|^{-1}\geq r_n. \end{equation} In particular, $I_n(r)\leq I_0(r)$. Thus by Theorem \ref{TWB}, $\psi_n$ is conformal at $0$, and hence $\phi_n$ is conformal at $\infty$. Thus we may normalize $\phi_n$ as claimed.
\end{rem}

\begin{thm}\label{limiting_proc} Let $(M_j)_{j=1}^\infty, (r_j)_{j=1}^\infty$ be strongly permissible, $c\in\mathbb{C}^\star$, and notation as in Definition \ref{entire_family}. Then for any $\varepsilon>0$, there exists $N_{\varepsilon}\in\mathbb{N}$ such that for $n\geq N_{\varepsilon}$, \begin{align} ||\phi_n(z)/z-1||_{L^{\infty}(\mathbb{C})} < \varepsilon.\end{align}
\end{thm}

\begin{proof} Briefly, this is a consequence of the error term in the conclusion (\ref{conclusion_of_TWB}) of Theorem \ref{TWB} depending only on $K$ and $I$, and not otherwise on the quasiconformal mapping under consideration. Let us explain further. We follow the notation of Remark \ref{normalization_remark}. For $I_n$ as in (\ref{I_n}), we have already noted that $I_n(r)\leq I_0(r)$. Moreover, $K(\psi_n)\leq K(\psi_0)$ by (\ref{D_n_relation}). Thus, by Theorem \ref{TWB}, there is a function $\iota: \mathbb{R}^{+} \rightarrow\mathbb{R}^{+}$ with  \begin{equation} \left| \frac{\psi_n(z)}{z} - 1 \right| < \iota(|z|) \end{equation} where $\iota(r)\rightarrow0$ as $r\rightarrow0$,  and $\iota$ does not depend on $n$. Thus, given $\varepsilon>0$, there exists $R>0$ such that: \begin{equation}\label{infinite_part} \left|\left| \frac{\phi_n(z)}{z}-1 \right|\right|_{L^\infty(|z|\geq R)} < \varepsilon \textrm{ for all } n\in\mathbb{N}. \end{equation} On the other hand, for all $z\in\mathbb{C}$ we have $D_n(z) \rightarrow 0$ as $n\rightarrow\infty$. Thus, by a standard argument, we have \begin{equation}\label{finite_part} ||\phi_n(z)- z||_{L^\infty(|z|\leq R)}\xrightarrow{n\rightarrow\infty} 0. \end{equation} The result now follows from (\ref{infinite_part}), (\ref{finite_part}), and a K\"obe distortion estimate for $\phi_n'(0)$. \end{proof}

We now turn to an application: we will show that as long as our parameters satisfy certain simple relations, the resulting entire function of Theorem \ref{mainthm} has a multiply connected wandering domain.

\begin{notation} For $0<\alpha<1$, we let $A_j^\alpha:=\{ z\in\mathbb{C} : \alpha^{-1} r_{j-1}\cdot\exp(\pi/M_{j-1})<|z|<\alpha r_j\}$.
\end{notation}

\begin{thm}\label{wandering_domain_application} Suppose $(M_j)_{j=1}^\infty$, $(r_j)_{j=1}^\infty$ are strongly permissible, $c\in\mathbb{C}^\star$, and \begin{equation}\label{sufficient} r_{j+1}=c_jr_j^{M_j}. \end{equation} Then, for any $0<\alpha<1$, there exists $N_\alpha$ such that, for $j\geq N_\alpha$, $A_j^\alpha$ is contained in a multiply-connected wandering domain for $f$ as in Theorem \ref{mainthm}.
\end{thm}

\begin{rem} The condition (\ref{sufficient}) is sufficient to guarantee a wandering domain for the function $f$ of Theorem \ref{mainthm}, but is far from necessary.
\end{rem}

\begin{proof} As usual, we denote $f=h\circ\phi^{-1}$. By Theorem \ref{mainthm}, there exists $N=N_\alpha\in\mathbb{N}$ such that \begin{equation}\label{phi_id} |\phi(z)/z-1| < |\sqrt{\alpha}-1| \textrm{ for } |z|\geq r_N. \end{equation} It follows that \begin{equation} \phi\left( A_j^\alpha \right) \subset A_j^{\sqrt{\alpha}} \textrm{ for } j\geq N. \end{equation} By (\ref{sufficient}), we have $h(r_j)=r_{j+1}$ for all $j$. Together with expansivity of $h$, this implies: \begin{equation} f\left( A_j^{\alpha} \right) \subset A_{j+1}^{\alpha} \textrm{ for } j\geq N,\end{equation} perhaps after increasing $N$. Thus the iterates of $f$ converge uniformly to $\infty$ on $A_j^{\alpha}$ for $j\geq N$, and so each such $A_j^{\alpha}$ is contained in a Fatou component for $f$, which we will call $\Omega_j$. 

If we suppose by way of contradiction that $\Omega_j=\Omega_k$ for some $j\not=k$, then $\Omega_j=\Omega_{j+1}$, and this would imply that $\Omega_{j}$ is unbounded. But $\Omega_j$ is multiply connected since $f(0)=0$ is an attracting fixed point, and this is a contradiction since all multiply connected Fatou components of $f$ must be bounded by \cite[Theorem~1]{MR0402044}. Thus, we may conclude that the $(\Omega_j)_{j=N}^\infty$ form a distinct sequence of Fatou components. Since $f(\Omega_j)\subset\Omega_{j+1}$, it follows that each such $\Omega_j$ is a wandering component for $f$. \end{proof}



\begin{rem} Let notation be as in Theorem \ref{wandering_domain_application}, and consider the family $f_n$ of Definition \ref{entire_family}. It is not difficult to show using Theorem \ref{limiting_proc} that if $0<\alpha<1$, then for sufficiently large $n$ we have that \emph{each} $A_j^{\alpha}$ for $j\in\mathbb{N}$ is contained in a wandering component $\Omega_j$ of $f_n$. 
\end{rem}

\end{document}